\documentclass[12pt]{amsart}

\author[A Branquinho]{Am\'\i lcar Branquinho\(^1\)}
\address{\(^1 \)CMUC, Departamento de Matem\'atica,
Universidade de Coimbra, 3001-454 Coimbra, Portugal}
\email{\(^1\)ajplb@mat.uc.pt}

\author[A Foulqui\'e]{Ana Foulqui\'e-Moreno\(^2\)}
\address{\(^2 \)CIDMA, Departamento de Matem\'atica, Universidade de Aveiro, 3810-193 Aveiro, Portugal}
\email{\(^2 \)foulquie@ua.pt}

\author[K Rampazzi]{Karina Rampazzi\(^3 \)}
\address{\(^3 \)Departamento de Matem\'atica, IBILCE, UNESP-Universidade Estadual Paulista, 15054-000, S\~ao Jos\'e do Rio Preto, SP, Brasil}
\email{\(^3 \)karina.rampazzi@unesp.br}

\usepackage{amsfonts}

\usepackage{textcomp}
\usepackage{charter}
\usepackage[]{bera}
\usepackage[charter]{mathdesign}
\usepackage{mathtools}

\usepackage{longtable}
\usepackage{enumerate}

\allowdisplaybreaks

\usepackage
[total={16cm,20cm},
top=4.25cm, 
lmargin=3cm, 
rmargin=3cm]
{geometry}

\usepackage[T1]{fontenc}
\usepackage[english]{babel}

\usepackage{epsfig}

\usepackage{hyperref}

\newtheorem{corollary}{Corollary}[section]
\newtheorem{lemma}{Lemma}[section]

\newtheorem{proposition}{Proposition}[section]

\newcommand{\T}{\ensuremath{\mathbb{T}}}

\newcommand{\dd}{\operatorname{d} \hspace{-.0425cm}}
\newcommand{\dde}{\operatorname e}
\newcommand{\ddi}{\operatorname i}

\usepackage{color} 

\title[Riemann--Hilbert problem for orthogonal polynomials in unit circle]{Generalized semiclassical orthogonal polynomials on the unit~circle:
A Riemann--Hilbert perspective}

\keywords{Orthogonal polynomials on the unit circle, Riemann--Hilbert Problems, modified Jacobi weight, modified Bessel weight, recurrence relations}

\subjclass[2020]{42C05, 33C45, 33C47, 47A56, 47A75}

\begin{document}

\begin{abstract}
In this work we show how to get advantage from the Riemann--Hilbert analysis in order to obtain first and second order differential equations
for the orthogonal polynomials and 
associated functions
with a weight on the unit circle.
We~deduce properties for the recurrence relation coefficients from differential properties of the weight. We take the so called generalized modified Jacobi and Bessel 
weights as a case~study.
\end{abstract}

\maketitle


\section{Introduction}\label{sec1}

The purpose of this work is to explore the connection between the theories of the Rie\-mann--Hilbert problem and of the orthogonal polynomials on the unit circle defined in terms of H\"older type weights. Furthermore, we show how Riemann--Hilbert analysis could help in deriving analytical properties for sequences of orthogonal polynomials associated with specific weight functions.

A weight function, \( w \), on the unit circle,
\(\mathbb{T}=\{z\in \mathbb{C} : |z|=1\}
 \), is called semiclassical if it satisfies the equation
\begin{align}\label{Eq-Tipo-Pearson-1-int}
\frac{\dd}{\dd \theta} \left( A (\dde ^{\ddi \theta}) w ( \theta ) \right) = B(\dde ^{\ddi \theta}) w( \theta ) 
, && \theta\in [0,2\pi],
\end{align}
where \(A \) and \(B \) are Laurent polynomials and \( A(\dde ^{\ddi \theta})=0 \) at the singular points of \( 1/w \). This definition is due to Magnus \cite{Ma00} in the case of the unit circle.
For the case of the real line, we refer to~\cite{Magnus_Painleve} and the references therein.
The equation \eqref{Eq-Tipo-Pearson-1-int} will be called a Pearson-type differential equation. In \cite{BranquinhoR2009} the Laguerre--Hahn families on the unit circle are treated, which include the semiclassical weight functions as particular cases.

If we allow in equation~\eqref{Eq-Tipo-Pearson-1-int} that the \( B \) can be taken as an entire function, we call \( w \) a \emph{generalized semiclassical weight}.

For orthogonality on the unit circle, the classification problem of semiclassical orthogonal polynomials was studied and divided into classes in \cite{CaSu97}, considering linear functionals that satisfy a Pearson-type equation.
This approach can be adapted to the equation~\eqref{Eq-Tipo-Pearson-1-int} considering that, if \(\deg A = p \) and \(\max\{p-1, \deg((p-1)A(z) +\ddi B(z))\}=q \), then \(w \) belongs to the class \( (p,q) \) of semiclassical weight functions.

A well known example of a semiclassical weight function on the unit circle is
\begin{align*}
w(\theta) =
 \dde ^{\ell \cos(\theta)} , && 0\leqslant \theta \leqslant 2\pi , &&
\text{with} && \ell>0 ,
\end{align*}
 see \cite{Is05}. 
The associated sequence of orthogonal monic polynomials, \( \{\Phi_n\}_{n\geqslant 0} \), on the unit circle satisfies the structure relation
\begin{align}\label{structurebessel}
\Phi_{n}^{\prime}(z) = n \Phi_{n-1}(z)
+ \frac{\ell \kappa_{n-2}^2}{2 \ \kappa_{n}^2} \Phi_{n-2}(z), && n \geqslant 2.
\end{align}
From this structure relation, it is possible to show the nonlinear difference equation
\begin{align} \label{Painleve}
\alpha_{n}(\ell) + \alpha_{n-2}(\ell) = -\frac{2n}{\ell} \frac{\alpha_{n-1}(\ell)}{1-\alpha_{n -1}^2(\ell)}, && n \geqslant 2.
\end{align}
where \( \alpha_{n-1}=-\overline{\Phi_n(0)} \) are known as Verblunsky coefficients.
This equation was presented by Periwal and Shevitz \cite{PS90}. It corresponds to a discrete Painlev\'e \({\rm dP_{II}} \) equation, see \cite{Va08}. Studies on semiclassical weight functions on the unit circle can also be found in \cite{BraRe12, Ma00} as well as in~\cite{Forrester_Witte,Ismail_Witte}.
In the last two papers the authors derive raising and lowering operators and find second order differential equations
for the orthogonal polynomials on the unit circle associated with 
semiclassical weights.
In~\cite{Forrester_Witte} it is stated a Riemann--Hilbert problem which is different to the one we present here, 
but the authors did not take advantage of it to derive differential properties for the polynomials and 
functions of second kind 
related with semiclassical weights.

The Riemann--Hilbert method is of great importance in the theory of boundary value problems for holomorphic functions. The main objective is to find a function that is holomorphic in a certain region, taking into account certain jump relations between the values of its limits over the points of a given contour. This approach to orthogonal polynomials was formulated by Fokas, Its, and Kitaev in \cite{FIK92}, and a method for calculating asymptotics with Riemann--Hilbert problems was presented by Deift and Zhou in \cite{De93}, first in the context of integrable systems. The most comprehensive application to orthogonal polynomials occurred in 1999, thanks to the work of Deift, Kriecherbauer, McLaughlin, Venakides, and Zhou, see \cite{De99, DKMVZ1, DKMVZ2}. These discoveries were strongly influenced by the relationship between orthogonal polynomials and the theory of random matrices \cite{CKV08, DIZ97}.

This technique has been widely used in the theory of orthogonal polynomials to obtain asymptotic and differential properties. In this work, we derive first and second order differential equations for generalized semiclassical orthogonal polynomials on the unit circle, using the Riemann--Hilbert techniques. We also show that some of the first-order differential equations are equivalent to structure relations found in \cite{BRS23}.
More details about the Riemann--Hilbert Problem can be found in \cite{ AF97, BI03, DIK11}.
Here we follow the Riemann--Hilbert Problem stated in~\cite{DIK11} (see also~\cite{Mf06}) where there appears two systems of associated functions that, so far as we know, have not been interpreted in terms of theirs differential properties.
For more recent results on Riemann--Hilbert theory related with matrix orthogonality, see \cite{BFFM22, BFFM, CM19, CCMa19}.

This work is structured as follows. 
In Section~\ref{sec2},
we start with the definition of orthogonality on the unit circle and present the Szeg\H{o} recurrence satisfied by these polynomials.
Next we state the Riemann--Hilbert problem and a technical lemma necessary for the development of the work. 
In Section~\ref{sec3} we study the analytic properties of the structure matrices associated with the generalized semiclassical weights.
We construct first and second order differential operators for the generalized semiclassical orthogonal polynomials in the unit circle and
associated function as well.

Furthermore, we determine, for the generalized semiclassical weights, the zero curvature formulas and find difference equations satisfied by Verblunsky coefficients.
We end this work, with Section \ref{sec4}, where by using the Riemann--Hilbert approach, we determine the structure matrices associated with the generalized modified, Bessel and Jacobi, weight functions,
and construct the first and second order differential operators for orthogonal polynomials and
associated functions
with these two considered weight functions on the unit circle.

\section{Riemann--Hilbert problem}\label{sec2}

Let \( \mu \) be a positive measure on the unit circle \( \mathbb{T} \),
 absolutely continuous, with respect to the Lebesgue measure, 
 i.e.
\begin{align*}
\dd \mu (z) = \frac{\nu (z)}{\ddi z} \, \dd z ,
&& \text{with} &&
w(\theta) = \nu (\dde^{\ddi \theta}) , && \theta\in[0,2\pi]. 
\end{align*}
The corresponding sequence of monic orthogonal polynomials,
\( \{\Phi_n\}_{n\geqslant 0} \),
satisfies 
\begin{align*}
 \int_{\mathbb{T}} \Phi_n (z)\overline{\Phi_m(z)} \dd \mu (z) 
= \kappa_n^{-2} \delta_{n,m}, && n,m \in \mathbb N \cup \{ 0 \},
\end{align*}
with \( \kappa_n>0 \), or equivalently,
 \begin{align} \label{eq:ortogonalidade}
\int_{\mathbb{T}} \Phi_n (z)\overline{\Phi_m(z)} \frac{\nu (z)}{\ddi z }\dd z
 =
\kappa_n^{-2} \delta_{n,m} , && n,m \in \mathbb N \cup \{ 0 \} .
\end{align}
 Studies about orthogonal polynomials on the unit circle can be found in~\cite{Is05, Simon-Book-p1}.
These polynomials satisfy the following relation, known as Szeg\H{o} recurrence, see \cite{Simon-Book-p1},
\begin{align}\label{Szegorecorrence}
\Phi_{n}(z) = z\Phi_{n-1}(z) - \overline{\alpha}_{n-1} \Phi_{n-1}^{*}(z), && n \in \mathbb N,
\end{align}
where \( \Phi_{0}(z) =1 \), \( \Phi_{n}^*(z) =z^n \overline{\Phi_{n}(1/\overline{z})} \) is a reciprocal polynomial and \( \alpha_{n-1} = - \overline{\Phi_{n}(0)} \), are the Verblunsky coefficients. Note that if the Verblunsky coefficients are known, then the relation \eqref{Szegorecorrence} generates the sequence of polynomials \( \{ \Phi_n \}_{n\geqslant 0} \). 

In terms of the powers of \( z \), and the sequence of polynomials \( \{ \Phi_n \}_{n \geqslant 0} \) and \( \{ \Phi_n^* \}_{n \geqslant 0} \), the orthogonality relations~\eqref{eq:ortogonalidade}
 read as
\begin{align*}
\int_{\mathbb{T}}z^{-k}\Phi_n(z)\dd \mu(z) 
&=0, && k=0,1,\ldots, n-1
&& \text{and} &&
\int_{\mathbb{T}}z^{-n}\Phi_n(z)\dd \mu(z) = \kappa^{-2}_n ,
 \\
\int_{\mathbb{T}}z^{-k}\Phi_n^*(z)\dd \mu(z)
& =0, && k=1,2,\ldots, n
&& \text{and} &&
\int_{\mathbb{T}} \Phi_n^* (z)\dd \mu(z) = \kappa^{-2}_n ,
\end{align*}
or, equivalently, as
\begin{align}\label{ortogcu}
\int_{\mathbb{T}}
\Phi_n(z) \frac{\nu(z)}{z^j} \dd z & =0, 
&& 
\int_{\T}
\Phi_{n}^*(z) \frac{\nu(z)}{z^{j+1}} \dd z
 = 0 , && j = 1 ,\ldots , n ,
 \\
\label{ortogcuultima}
\int_{\mathbb{T}}
\Phi_n(z) \frac{\nu(z)}{\ddi z^{n+1}} \dd z &= \kappa_n^{-2}, 
&& 
\int_{\T}
\Phi_{n}^*(z) \frac{\nu(z)}{\ddi z} \dd z
 = \kappa_n^{-2} , && n \in \mathbb N 
 .
\end{align}
From Szeg\H{o} recurrence, we have the following equations
\begin{align}
\label{recszego}
\Phi_{n+1}^{*}(z) &= \Phi_{n}^{*}(z) - \alpha_{n} z \Phi_{n}(z), && n\geqslant 0, \\
\nonumber\Phi_{n+1}^{*}(z) &= (1-|\alpha_{n} |^2) \Phi_{n}^{*}(z)-\alpha_{n} \Phi_{n+1}(z), && n\geqslant 0,
\end{align}
and \( \kappa_{n}^{2}/\kappa_{n+1}^{2} = 1-|\alpha_n|^2 \), see \cite{Simon-Book-p1}.
We denote the monic orthogonal polynomials on the unit circle as
\begin{align}\label{newphin}
 \Phi_n(z)=z^n+\Phi_1^{n}z^{n-1}+\Phi_2^{n}z^{n-2}+\cdots+\Phi_{n-1}^{n}z+\Phi_n^n, && n\geqslant 0.
\end{align}
Applying the Szeg\H{o} recurrence~\eqref{Szegorecorrence}, it is easy to see that the coefficients \( \Phi_{1}^n \) in \eqref{newphin} are given in terms of the Verblunsky coefficients as 
\begin{align*}
\Phi_{1}^n = \sum_{j=0}^{n-1} \overline{\alpha}_{j} \alpha_{j-1}, 
&& n \geqslant 1 && \text{with} && \alpha_{-1} = -1 .
\end{align*}
Let \( \nu \) be a H\"older type weight on the unit circle~\(\mathbb T \), oriented in the positive direction.
Then,~it can be seen (cf.~\cite{DIK11,Mf06}),
that the unique solution of the Riemann--Hilbert problem, which consists in the determination of a \(2 \times 2 \) matrix complex function such~that:
\begin{enumerate}[\rm \bf (RH1)]

\item
\( Y_n \) is holomorphic in \( \mathbb{C}\setminus\mathbb{T} \);

\item
 \( Y_n \) satisfies the jump condition for all \( t\in \mathbb{T} \),
\begin{align*}
 {(Y_n})_{+}(t)={(Y_n)}_{-}(t)
\begin{bmatrix}
1 & \nu(t)/t^n \\
0 & 1
\end{bmatrix},
\end{align*}
where
\( Y_n \) has continuous boundary values \(Y_+(z) \) as \(z \)
approaches the unit circle from the inside, and \(Y_-(z) \) from the outside;

\item
\( Y_n \) has the following asymptotic behavior at infinity
\begin{align*}
Y_n(z)
= 
\Big( \mathbf{\operatorname I}_2 + \operatorname{O} ( z^{-1} ) \Big)
\begin{bmatrix}
z^{n} & 0\\
0 & z^{-n}
\end{bmatrix}
&&
\text{as} && z\to\infty
\end{align*}
\end{enumerate}
(the \( \mathbf{\operatorname I}_2 \) is the order \( 2 \) identity matrix) is given by
\begin{align*}
Y_n(z)
=
\begin{bmatrix}
\Phi_n(z) & 
 \displaystyle
 \frac{1}{2\pi}\int_\mathbb{T}\frac{\Phi_n(t)}{t-z}\frac{\nu(t)}{\ddi t^n}\dd t
 \\[.375cm]
- 2 \pi\kappa_{n-1}^2 \Phi^*_{n-1}(z) & 
\displaystyle
- \kappa_{n-1}^2\int_\mathbb{T} \frac{\Phi^*_{n-1}(t)}{t-z}
\frac{\nu(t)}{\ddi t^n}\dd t
\end{bmatrix},
&& n \in \mathbb N .
\end{align*}
Moreover, it is also proved that \(\det Y_n = 1 \), \( n \in \mathbb N \).

We will write, \(Y_n \), 
in the form
\begin{align}\label{ynnew}
Y_n(z)=\begin{bmatrix}
\Phi_n(z) & G_n(z) \\[.075cm]
-b_{n-1}\Phi^*_{n-1}(z) & -b_{n-1} G^*_{n-1}(z)
\end{bmatrix} ,
\end{align}
where
\begin{align}
\label{eq:bbconst}
b_{n-1}= 2\pi\kappa_{n-1}^2 ,
&& n \in \mathbb N ,
\end{align} 
and
\begin{align}
\label{G_n}
 G_n(z) & = \frac{1}{2\pi \ddi }\int_\mathbb{T}
\frac{\Phi_n(t)}{t-z}\frac{\nu(t)}{t^n}
\dd t, 
&&
 G^*_{n-1}(z)
 =\frac{1}{2\pi \ddi }\int_\mathbb{T} \frac{\Phi^*_{n-1}(t)}{t-z}
\frac{\nu(t)}{t^n}\dd t ,
\end{align}
 will be called 
associated functions with the orthogonal polynomials, \( \Phi_n \) and weight function \( \nu \) on the unit circle.
Again, since \( \det Y_n=1 \),
we can write
\begin{align}\label{inversaynnew}
Y_n^{-1}(z)=\begin{bmatrix}
-b_{n-1} G^*_{n-1}(z) & -G_n(z) \\[.075cm]
b_{n-1}\Phi^*_{n-1}(z) & \Phi_n(z)
\end{bmatrix}.
\end{align}
Now, we will work on the 
associated functions
\( G_n \) and \( G_{n-1}^* \).

\begin{proposition}
Let \( 
\dd \mu (z) = 
\frac{\nu (z)}{\ddi z} \dd z
 \) be a measure defined on the unit circle. Then their sequences of
associated functions
defined by~\eqref{G_n} 
satisfies the following recurrence relations,
\begin{align}
G_{n}(z) &= G_{n-1}(z) - \overline{\alpha}_{n-1} G^*_{n-1}(z), && n \in \mathbb N,
\label{eq:G_n} \\
z G_{n}^* (z) &= G_{n-1}^* (z) - \alpha_{n-1} G_{n-1}(z) , && n \in \mathbb N . 
\label{eq:G_n^*}
\end{align}
\end{proposition}

\begin{proof} To show the equation \eqref{eq:G_n} we multiply \eqref{Szegorecorrence}
in \( t \) variable by \( \frac{\nu(t)}{ 2 \pi \ddi t^{n+1}(t-z)} \) and then take integral in~\(\mathbb T \) with respect to the variable \( t \), i.e.
\begin{align*}
\frac{1}{2\pi \ddi } \int_\mathbb{T}\frac{\Phi_{n+1}(t)\nu(t)}{t^{n+1}(t-z)}{\dd t}= \frac{1}{2\pi \ddi } \int_\mathbb{T}\frac{t \Phi_{n}(t)\nu(t)}{t^{n+1}(t-z)} \dd t - \overline{\alpha}_{n} \frac{1}{2\pi \ddi } \int_\mathbb{T}\frac{\Phi^*_{n}(t)\nu(t)}{t^{n+1}(t-z)}{\dd t}
\end{align*}
which is equivalent to~\eqref{eq:G_n}.
To show \eqref{eq:G_n^*}, we proceed in the same way as before, now using the relation \eqref{recszego} in \( t \) variable and multiplying by \( \frac{\nu(t)}{ 2 \pi \ddi t^{n}(t-z)} \), i.e.
\begin{align*}
\frac{1}{2\pi \ddi } \int_\mathbb{T}\frac{\Phi^*_{n}(t)\nu(t)}{ t^{n}(t-z)}{\dd t}= \frac{1}{2\pi \ddi } \int_\mathbb{T}\frac{t \Phi_{n-1}(t)\nu(t)}{t^{n}(t-z)}{\dd t}- \alpha_{n-1} \frac{1}{2\pi \ddi } \int_\mathbb{T}\frac{t \Phi_{n-1}(t)\nu(t)}{ t^{n}(t-z)}{\dd t}.
\end{align*}
Now, taking into account that
\begin{align*}
& \frac{1}{2\pi \ddi } \int_\mathbb{T}\frac{\Phi^*_{n}(t)\nu(t)}{t^{n}(t-z)} \dd t
 = 
\frac{1}{2\pi \ddi } \int_\mathbb{T}\frac{t-z+z}{t-z}\frac{\Phi^*_{n}(t)\nu(t)}{ t^{n+1}}\dd t \\
 & \hspace{1.5cm} = \frac{1}{2\pi \ddi } \int_\mathbb{T}\frac{\Phi^*_{n}(t)\nu(t)}{t^{n+1}}{\dd t}+ \frac{z}{2\pi \ddi } \int_\mathbb{T}\frac{\Phi^*_{n}(t)\nu(t)}{t^{n+1}(t-z)}{\dd t}
 && \text{and by~\eqref{ortogcu}
 }
 \\
 & \hspace{1.5cm} =
 \frac{z}{2\pi \ddi } \int_\mathbb{T}\frac{\Phi^*_{n}(t)\nu(t)}{ t^{n+1}(t-z)}\dd t.
\end{align*}
So, we obtain~\eqref{eq:G_n^*}.
\end{proof}

Using the orthogonality properties, in the next result we present the asymptotic expansion of \( G_n \) and \( G_{n-1}^* \) about \( \infty \).

\begin{lemma}[Technical]
\label{lemma6}
Let \( \dd \mu (z) = 
\frac{\nu (z)}{\ddi z} \dd z \) be a measure defined on the unit circle.
Then, 
\begin{align}\label{G_n_0}
G_n(0)=\frac{1}{b_n} ,
 && 
G_{n-1}^*(0)=\frac{\alpha_{n-1}}{b_{n-1}},
\end{align}
and
\begin{align}
 \label{Gninf}
G_n(z)
 & =
-\frac{\overline{\alpha}_n}{b_n}\frac{1}{z^{n+1}}+\left(\frac{\overline{\alpha}_n}{b_n}\Phi_{1}^{n+1}-\frac{\overline{\alpha}_{n+1}}{b_{n+1}}\right)\frac{1}{z^{n+2}}+\cdots , && n \geqslant 0 ,
 \\
\label{Gnestrelainf}
G_{n-1}^*(z) 
 & =
\frac 1{b_{n-1}} \left(-\frac{1}{z^n}+\frac{\Phi_{1}^{n}}{z^{n+1}}-\frac{\Phi_{1}^{n}\Phi_{1}^{n+1}-\Phi_{2}^{n+1}}{z^{n+2}} + \cdots \right) , && n \geqslant 1 .
\end{align}
\end{lemma}

\begin{proof}
Taking \( z=0 \) in \eqref{G_n}, we have
\begin{align*}
G_{n}(0) &= \frac{1}{2\pi \ddi }\int_\mathbb{T}
\frac{\Phi_n(t)}{t^{n+1}}\nu(t)
\dd t 
 .
\end{align*}
Using~\eqref{ortogcuultima} and~\eqref{eq:bbconst}
we get the first identity in \eqref{G_n_0}.
Now, 
taking \( z=0 \) in
the second identity in~\eqref{G_n_0}
and using \eqref{recszego} together with the orthogonality properties, we get
\begin{align*}
G^*_{n-1}(0)&=\frac{1}{2\pi \ddi }\int_\mathbb{T} \frac{ \Phi_n^*(t)+\alpha_{n-1}t\Phi_{n-1}(t) }{t^{n+1}} \nu(t) \dd t .
\end{align*}
Therefore, using~\eqref{ortogcuultima},~\eqref{ortogcu}, and~\eqref{eq:bbconst}
we get the second identity in~\eqref{G_n_0}.
To prove equation~\eqref{Gnestrelainf}, we take into account the second identity in~\eqref{G_n}
in order to get
\begin{align*}
G^*_{n-1}(z) & = \frac{1}{2\pi \ddi }\int_\mathbb{T}\frac{\Phi^*_{n-1}(t)\nu(t)}{t^n(t-z)}\dd t \\
 & = -\sum_{j=0}^{\infty}\frac{1}{z^{j+1}}\frac{1}{2\pi}\int_\mathbb{T}t^j\Phi_{n-1}^*(t)\frac{\nu(t)}{\ddi t^n}\dd t 
 && \text{and by~\eqref{ortogcu}} \\
 & = -\sum_{j=0}^{\infty}\frac{1}{z^{j+n}}\frac{1}{2\pi}\int_\mathbb{T}t^j\Phi_{n-1}^*(t)\frac{\nu(t)}{\ddi t}\dd t .
 \end{align*}
Now, by \eqref{ortogcuultima} we get that the coefficient of the \( (-n) \)-th \( z \) power is \( -b_{n-1}^{-1} \).
To determine the coefficient of the \( (-n-1) \)-th \( z \) power, we begin by identify
\begin{align*}
\int_\mathbb{T} \Phi_{n-1}^*(t)\frac{\nu(t)}{\ddi}\dd t=
\int_\mathbb{T}t^{n}\overline{\Phi_{n-1}(t)}\frac{\nu(t)}{\ddi t}\dd t.
\end{align*}
Hence plugging into the orthogonal relation
the expression of \( \Phi_n \) given in~\eqref{newphin}
\begin{align*}
\int_\mathbb{T}
\Phi_{n}(t)\overline{\Phi_{n-1}(t)}\frac{\nu(t)}{\ddi t}\dd t 
=
\int_\mathbb{T}
\big( t^n+\Phi_1^{n}t^{n-1}+\Phi_2^{n}t^{n-2}+\cdots \big) \overline{\Phi_{n-1}(t)}\frac{\nu(t)}{\ddi t}\dd t 
= 0
\end{align*}
we arrive, taking into account the orthogonality condition~\eqref{eq:ortogonalidade}, to the equation
\begin{align*}
\int_\mathbb{T}t^{n}\overline{\Phi_{n-1}(t)}\frac{\nu(t)}{\ddi t}\dd t 
+
\Phi_1^{n}
\int_\mathbb{T}
t^{n-1} \overline{\Phi_{n-1}(t)}\frac{\nu(t)}{\ddi t}\dd t 
= 0 ,
\end{align*}
and so the coefficient of the \( (-n-1) \)-th \( z \) power is
 \( \Phi_1^{n} b_{n-1}^{-1} \).
To determine the third coefficient,
\begin{align*}
\int_\mathbb{T} t \Phi_{n-1}^*(t)\frac{\nu(t)}{\ddi}\dd t=
\int_\mathbb{T}t^{n+1}\overline{\Phi_{n-1}(t)}\frac{\nu(t)}{\ddi t}\dd t.
\end{align*}
 Again, using the same argument, i.e.,
\begin{align*}
\int_\mathbb{T}
\Phi_{n+1}(t)\overline{\Phi_{n-1}(t)}\frac{\nu(t)}{\ddi t}\dd t 
=
\int_\mathbb{T}
\big( t^{n+1}+\Phi_1^{n+1}t^{n}+\Phi_2^{n+1}t^{n-1}+\cdots \big) \overline{\Phi_{n-1}(t)}\frac{\nu(t)}{\ddi t}\dd t 
= 0 ,
\end{align*}
we arrive, taking into account the orthogonality condition~\eqref{eq:ortogonalidade}, to the equation
\begin{align*}
\int_\mathbb{T}t^{n+1}\overline{\Phi_{n-1}(t)}\frac{\nu(t)}{\ddi t}\dd t
&=-\Phi_1^{n+1}\int_\mathbb{T}t^{n}\overline{\Phi_{n-1}(t)}\frac{\nu(t)}{\ddi t}\dd t-\Phi_2^{n+1}\int_\mathbb{T}t^{n-1}\overline{\Phi_{n-1}(t)}\frac{\nu(t)}{\ddi t}\dd t
\end{align*}
and so the coefficient of the \( (-n-2) \)-th \( z \) power is
\( (\Phi_1^{n+1}\Phi_1^n-\Phi_2^{n+1}) b_{n-1}^{-1} \).
Therefore, in general, the coefficients of the \( z \) powers in the expansion~\eqref{Gnestrelainf}
are given by 
\begin{align*}
\int_\mathbb{T}t^{n+j}\overline{\Phi_{n-1}(t)}\frac{\nu(t)}{\ddi t}\dd t
 & =
-\sum_{k=1}^{j+1}\Phi_{k}^{n+j} \int_\mathbb{T} t^{n-k+j}\overline{\Phi_{n-1}(t)}\frac{\nu(t)}{\ddi t}\dd t ,
j \in \mathbb N \cup \{ 0 \} .
\end{align*}
 From \eqref{G_n}, we have
\begin{align*}
G_{n}(z) & = \frac{1}{2\pi}\int_\mathbb{T}\frac{\Phi_{n}(t)\nu(t)}{\ddi t^n(t-z)}\dd t
 = -\sum_{j=0}^{\infty}\frac{1}{z^{j+1}}\frac{1}{2\pi}\int_\mathbb{T}t^j\Phi_{n}(t)\frac{\nu(t)}{\ddi t^n}\dd t
 \\ 
 &= -\sum_{j=0}^{\infty}\frac{1}{z^{n+j+1}}\frac{1}{2\pi}\int_\mathbb{T}t^{j+1}\Phi_{n}(t)\frac{\nu(t)}{\ddi t}\dd t .
 \end{align*} 
For the first coefficients, we observe that
\begin{align}\label{eqauxformgeral}
 \int_{\mathbb{T}} t\Phi_{n}(t) \frac{\nu(t)}{\ddi t}\dd t = \frac{\overline{\alpha}_{n}}{\kappa_n^{2}} .
\end{align}
In fact, using the Szeg\H{o} recurrence \eqref{Szegorecorrence}, we obtain
\begin{align*}
 \int_{\mathbb{T}} t\Phi_{n}(t) \frac{\nu(t)}{\ddi t}\dd t
 & = \int_{\mathbb{T}}\Phi_{n+1}(t)\frac{\nu(t)}{\ddi t}\dd t + \overline{\alpha}_{n} \int_{\mathbb{T}} \Phi_{n}^{*}(t)\frac{\nu(t)}{\ddi t}\dd t ,
\end{align*}
and taking into account~\eqref{ortogcu} and \eqref{ortogcuultima},
we arrive to~\eqref{eqauxformgeral}.
Now, proceeding as before departing from
\begin{align*}
 \int_{\mathbb{T}} t^2 \Phi_{n}(t) \frac{\nu(t)}{\ddi t}\dd t
 & = \int_{\mathbb{T}} t \Phi_{n+1}(t)\frac{\nu(t)}{\ddi t}\dd t + \overline{\alpha}_{n} \int_{\mathbb{T}} t \Phi_{n}^{*}(t)\frac{\nu(t)}{\ddi t}\dd t ,
\end{align*}
and using the previous calculations, we get 
\begin{align*}
 \int_{\mathbb{T}} t^2\Phi_{n}(t) \frac{\nu(t)}{\ddi t}\dd t&=\frac{\overline{\alpha}_{n+1}}{\kappa_{n+1}^2}-\Phi_{1}^{n+1}\frac{\overline{\alpha}_n}{\kappa_n^2}.
\end{align*}
Therefore, in general, we have the relation
\begin{align*}
 \int_{\mathbb{T}} t^{j+1}\Phi_{n}(t) \frac{\nu(t)}{\ddi t}\dd t
 & =\frac{\overline{\alpha}_{n+j}}{\kappa_{n+j}^2}-\overline{\alpha}_n\int_{\mathbb{T}} t^{n+j}\overline{\Phi_{n}(t)} \frac{\nu(t)}{\ddi t}\dd t.
\end{align*}
In this way, we get the asymptotic expansion for \( G_n \) at infinity~\eqref{Gninf}.
\end{proof}

Now, we apply the Riemann--Hilbert problem to derive recurrence relations for orthogonal polynomials and
associated functions.

\begin{proposition}
 Let \( Y_n \) be the solution of the Riemann--Hilbert problem, then for all \( n \in \mathbb N \), 
\begin{align}\label{tn}
Y_{n+1}(z)
 \begin{bmatrix}
1 & 0\\
0 & z
\end{bmatrix}
=T_n(z)Y_n(z),
 && \text{where} &&
T_n(z)=\begin{bmatrix}
z + \overline{\alpha}_n\alpha_{n-1} & 
\displaystyle
{ \overline{\alpha}_n}/{b_{n}}
 \\
\alpha_{n-1} b_n & 1
\end{bmatrix}
\end{align}
which is called the transfer matrix.
\end{proposition}

\begin{proof} The matrices \( Y_n \) and \( Y_n^{-1} \) are holomorphic at \( \mathbb{C}\setminus\mathbb{T} \), then
the matrix function
\begin{align*}
T_n(z)
=
Y_{n+1}(z)
 \begin{bmatrix}
1 & 0\\
0 & z
\end{bmatrix} 
Y_n^{-1}(z)
\end{align*}
 is holomorphic at \( \mathbb{C}\setminus\mathbb{T} \). 
However, \( T_n \) has no jump on the unit circle \( \mathbb{T} \). In fact, we have
\begin{align*}
 (T_n)_{+}(z) &=(Y_{n+1})_{+}(z)\begin{bmatrix}
1 & 0\\
0 & z
\end{bmatrix}(Y_n^{-1})_{+}(z)\\
 &= {(Y_{n+1})}_{-}(z)
\begin{bmatrix}
1 & \nu(z)/z^{n+1}\\
0 & 1
\end{bmatrix} \begin{bmatrix}
1 & 0\\
0 & z
\end{bmatrix}
\begin{bmatrix}
1 & -\nu(z)/z^{n}\\
0 & 1
\end{bmatrix} {(Y_{n}^{-1})}_{-}(z)\\
 &= {(Y_{n+1})}_{-}(z)
 \begin{bmatrix}
1 & 0\\
0 & z
\end{bmatrix}
 {(Y_{n}^{-1})}_{-}(z) = (T_n)_{-}(z), \quad z\in\mathbb{T}.
\end{align*}
Therefore, \( T_n \) is holomorphic in \( \mathbb{C} \) and it is an entire function.
Moreover, from the representation of \( Y_n \) and \( Y_n^{-1} \) given by~\eqref{ynnew} and~\eqref{inversaynnew}, respectively,
for \( z\in \mathbb C \setminus \mathbb{T} \), we obtain
\begin{align*}
 T_n(z) 
 &=\begin{bmatrix}
\Phi_{n+1}(z) & G_{n+1}(z) \\[.05cm]
-b_{n}\Phi^*_{n}(z) & -b_{n} G^*_{n}(z)
\end{bmatrix}
\begin{bmatrix}
1 & 0\\[.05cm]
0 & z
\end{bmatrix}
\begin{bmatrix}
-b_{n-1} G^*_{n-1}(z)& -G_{n}(z) \\[.05cm]
b_{n-1}\Phi^*_{n-1}(z) & \Phi_{n}(z)
\end{bmatrix}\\
&=\begin{bmatrix}
b_{n-1}(z\Phi_{n-1}^*(z)G_{n+1}(z)-\Phi_{n+1}(z) G^*_{n-1}(z))
 & -\Phi_{n+1}(z)G_{n}(z)+z\Phi_n(z)G_{n+1}(z) \\[.1cm]
b_nb_{n-1}(\Phi_n^*(z)G_{n-1}^*(z)-z\Phi^*_{n-1}(z)G_n^*(z)) & b_n(\Phi_n^*(z)G_n(z)-z\Phi_{n}(z)G_n^*(z))
\end{bmatrix} .
\end{align*}
Now, using \eqref{Gninf} and \eqref{Gnestrelainf}, and taking into account the behavior at infinity, and 
using Liouville's Theorem, we get
the explicit expression for \( T_n \) in~\eqref{tn}.
\end{proof}

As a direct consequence from equation \eqref{tn} we obtain the recurrence relations for the entries in \( Y_n \).

\begin{corollary}
 Let \( Y_n \) be the solution of the Riemann--Hilbert problem,
then for all \(n \in \mathbb N \),
\begin{align*}
 \Phi_{n+1}(z)&=(z+\overline{\alpha}_n\alpha_{n-1})\Phi_n(z)-\frac{b_{n-1}}{b_n}\overline{\alpha}_n\Phi_{n-1}^*(z) , \\
 -b_{n}\Phi_n^{*}(z)&=\alpha_{n-1}b_n\Phi_n(z)-b_{n-1}\Phi_{n-1}^*(z) , \\
 zG_{n+1}(z)&=(z+\overline{\alpha}_n\alpha_{n-1})G_n(z)-\frac{b_{n-1}}{b_n}\overline{\alpha}_nG_{n-1}^*(z) , \\
 -b_nzG_n^*(z)& =\alpha_{n-1}b_nG_{n}(z)-b_{n-1}G_{n-1}^*(z) .
\end{align*}
\end{corollary}

\section{Fundamental matrices from Riemann--Hilbert problem}\label{sec3}

We introduce the constant jump fundamental matrix,~\( Z_n \), which will be instrumental in what follows 
\begin{align}\label{Zn}
 Z_n(z)=Y_n(z)C_n(z) ,
 && \text{where} 
 &&
 C_n(z)
 ={\begin{bmatrix}
\frac{z^{-n/2}}{(\nu(z))^{-1/2}} & 0 \\
0 & \frac{z^{n/2}}{(\nu(z))^{1/2}}
\end{bmatrix}},
&& n \in \mathbb N
 .
\end{align}

\begin{proposition}\label{propnewZn}
Let \( Y_n \) be the solution of the Riemann--Hilbert problem, then
\begin{align*}
 \sqrt{z} \, Z_{n+1}(z) = T_n(z) \, Z_n(z), && n \in \mathbb N
 ,
\end{align*}
where \( T_n \) is defined in~\eqref{tn}.
\end{proposition}

\begin{proof} 
It follows from direct computations, using~\eqref{tn} and~\eqref{Zn}.
\end{proof}

Now, we introduce the structure matrix, given in terms of the logarithmic derivatives of the fundamental matrix defined in~\eqref{Zn}
\begin{align}\label{Mn}
 M_n(z)=Z_n^\prime(z)Z_n^{-1}(z), && n \in \mathbb N
 .
\end{align}
We study the analytic properties of the functions just defined, for the case of \emph{generalized semiclassical weights} \( \nu \), i.e. weights \( \nu \) satisfying
\begin{align}
\label{eq:pearson_nu}
z A(z) \nu^\prime(z) =
q(z) \nu(z) ,
\end{align}
with \( A\) and \( q \) are a given polynomial and an entire function, respectively,
and the differential equation is taken off the set of the zeros of the polynomial \( z A(z) \), that will be called~\( \mathcal{Z}_A \).

Note that, when \( q(z) = B (z) - z A^\prime (z) \), where \(B \) is a polynomial, \( w \) satisfies~\eqref{Eq-Tipo-Pearson-1-int} and so is a semiclassical weight.


\begin{proposition}\label{teo48} 
Let \( \nu (z) 
 \) be a
generalized 
semiclassical weight (cf.~\eqref{eq:pearson_nu}), then the matrix function \( Z_n \) definined by~\eqref{Zn} is 
such that:
\begin{enumerate}[\rm i)]
\item
\( Z_n \) is holomorphic in \( \mathbb{C}\setminus ( \mathbb{T} \cup \mathcal{Z}_A ) \);
\item
\( Z_n \) satisfies the constant jump condition
\begin{align*}
{(Z_n})_{+}(t)={(Z_n)}_{-}(t)
\begin{bmatrix}
1 & 1\\
0 & 1
\end{bmatrix}, && t\in\mathbb{T};
\end{align*}
\item
\( Z_n \) has the following asymptotic behavior at infinity
\begin{align*}
Z_n(z)
= 
\Big( \mathbf{\operatorname I}_2 + \operatorname{O} \big( \frac 1 z \big) \Big)
\begin{bmatrix}
z^{-n/2}\nu^{-1/2}(z) & 0\\
0 & z^{n/2}\nu^{1/2}(z) 
\end{bmatrix}
&&
\text{as} && z\to\infty .
\end{align*}
\end{enumerate}
\end{proposition} 

\begin{proof}
\begin{enumerate}[\rm i)]
\item
Since \( C_n \) is a matrix of entire functions in \( \mathbb{C}\setminus ( \mathbb{T}\cup \mathcal{Z}_A ) \) and \( Y_n \) is a holomorphic matrix in \( \mathbb{C }\setminus\mathbb{T} \), then \( Z_n \) is holomorphic at \( \mathbb{C}\setminus \big( \mathbb{T}\cup \mathcal{Z}_A \big) \).
\item
For \( t\in\mathbb{T} \), we have
\begin{align*}
{(Z_n})_{+}(t) &= {(Y_n})_{+}(t){\begin{bmatrix}
\frac{t^{-n/2}}{\nu^{-1/2}} & 0 \\
0 & \frac{t^{n/2}}{\nu^{1/2}}
\end{bmatrix}} 
 = {(Y_n)}_{-}(t)
\begin{bmatrix}
1 & \frac{\nu(t)}{t^n}\\
0 & 1
\end{bmatrix}{\begin{bmatrix}
\frac{t^{-n/2}}{\nu^{-1/2}} & 0 \\
0 & \frac{t^{n/2}}{\nu^{1/2}}
\end{bmatrix}}\\
& = {(Y_n)}_{-}(t)
\begin{bmatrix}
\frac{t^{-n/2}}{\nu^{-1/2}} & \frac{\nu^{1/2}}{t^{n/2}} \\
0 & \frac{t^{n/2}}{\nu^{1/2}}
\end{bmatrix} 
 = {(Y_n)}_{-}(t)
{\begin{bmatrix}
\frac{t^{-n/2}}{\nu^{-1/2}} & 0 \\
0 & \frac{t^{n/2}}{\nu^{1/2}}
\end{bmatrix}}\begin{bmatrix}
1 & 1\\
0 & 1
\end{bmatrix}\\
& = {(Z_n)}_{-}(t)
\begin{bmatrix}
1 & 1\\
0 & 1
\end{bmatrix}.
\end{align*}
\end{enumerate}

{\noindent}iii) It follows from the asymptotic behavior of the matrix \( Y_n \) at infinity. 
\end{proof}

\begin{proposition}
\label{teo:geral}
Let \( \nu(z) \) be a
generalized
semiclassical weight.
Then, the corresponding structure matrix~\( M_n \), \( n\in\mathbb{N} \), given in~\eqref{Mn}, is a holomorphic function in \( \mathbb {C} \setminus \mathcal{Z}_A \).
\end{proposition}

\begin{proof}
 Since \( Z_n \) is holomorphic in \( \mathbb{C}\setminus( \mathbb{T}\cup\mathcal{Z}_A ) \), then
\( M_n(z)=Z_n^\prime (z)Z_n^{-1}(z) \) is also holomorphic in \( \mathbb{C}\setminus( \mathbb{T}\cup\mathcal{Z}_A ) \). 
Due to the fact that \( Z_n \) has a constant jump on \( \mathbb{T}\cup \mathcal{Z}_A \), the~matrix function \( Z_n^\prime \) has the same constant jump on \( \mathbb{T}\cup \mathcal{Z}_A \), so the matrix \( M_n \) has no jump on \( \mathbb{T}\cup \mathcal{Z}_A \), and it follows that in each element of \( \mathcal Z_A \), the functions \( M_n \) have an isolated singularity. 
\end{proof}

The next proposition is very important for finding nonlinear difference equations satisfied by the Verblunsky coefficients.

\begin{proposition}\label{propnula} 
Let \( \nu(z) \) be a
generalized
semiclassical weight.
Then, the corresponding structure matrix~\( M_n \), \( n\in\mathbb{N} \), given in~\eqref{Mn}
 satisfies
\begin{align}\label{curvatureformula}
T_n^\prime - \frac 1 {2z} T_n
 =
{M}_{n+1}T_n -T_n {M}_n ,
&& n \in \mathbb N ,
\end{align}
which is a zero curvature formula.
\end{proposition}

\begin{proof}
From Proposition~\ref{propnewZn} we know that \( \sqrt{z} \, Z_{n+1}=T_nZ_n \), then taking derivative
 and again using 
Proposition~\ref{propnewZn}, we obtain
\begin{align*}
\Big\{ T_n^\prime -\frac{1}{2z}T_n +T_n{M}_n-{M}_{n+1}T_n \Big\} Z_n
=\mathbf{0}_{2}.
\end{align*}
Therefore, since \( Z_n \) is invertible, 
the result follows.
\end{proof}

We also have a kind of reciprocal of Proposition~\ref{propnula}.

\begin{proposition}\label{propeqfund1} 
Let \( \nu \) be a H\"older type weight defined on \( \mathbb T \), with sequences of transfer matrices \( \{ T_n \}_{n \in \mathbb N} \) and constant jump fundamental matrices \( \{ Z_n \}_{n \in \mathbb N} \).
If \( M_n \), \( n\in\mathbb{N} \), satisfies~\eqref{curvatureformula}, with
 \( M_0(z) \coloneqq Z_0^\prime(z) Z_0^{-1}(z)\). Then,
\( \{ M_n \}_{n \in \mathbb N} \) satisfies~\eqref{Mn}. 
\end{proposition}

\begin{proof}
Multiplying~\eqref{curvatureformula} by \( Z_n \), and taking into account that
\( \sqrt{z} \, Z_{n+1}=T_nZ_n \), we arrive to
\begin{align*}
\sqrt{z} \left( Z^\prime_{n+1} - M_{n+1} Z_{n+1} \right) =T_n \left( Z^\prime_{n} - M_{n} Z_{n} \right) ,
&&
n \in \mathbb N.
\end{align*}
Multiplying by \( (\sqrt{z})^n \) and iterating the procedure we arrive to
\begin{align*}
(\sqrt{z})^{n+1} \left( Z^\prime_{n+1} - M_{n+1} Z_{n+1} \right)
 = T_n T_{n-1} \cdots T_0 \left( Z^\prime_{0} - M_{0} Z_{0} \right) ,
\end{align*}
and taking into account our hypothesis we get that~\eqref{Mn} takes place.
\end{proof}

Next, we state a second curvature formula.

\begin{proposition}\label{propeqfund2} 
Let \( \nu(z) \) be a
generalized
semiclassical weight on \( \mathbb T \).
Then, the corresponding structure matrix~\( M_n \), \( n\in\mathbb{N} \), given in~\eqref{Mn}
 satisfies
\begin{align*}
{M}_{n+1}\Big\{T_n^\prime -\frac{1}{2z}T_n\Big\}+\Big\{T_n^\prime-\frac{1}{2z}T_n\Big\}{M}_{n}={M}_{n+1}^2T_n-T_n{M}_n^2.
\end{align*}
\end{proposition}

\begin{proof}
Multiplying on the left of \eqref{curvatureformula} by \( {M}_{n+1} \), we get
\begin{align}\label{eq1}
{M}_{n+1} T_n^\prime-{M}_{n+1}\frac{1}{2z}T_n={M}_{n+1}^2T_n-{M}_{n+1}T_n{M}_n
\end{align}
and on the other hand, multiplying on the right by \( {M}_n \) we obtain
\begin{align}\label{eq2}
T_n^\prime {M}_n-{M}_{n+1}\frac{1}{2z}T_n {M}_n={M}_{n+1}^2T_n {M}_n- {M}_{n+1}T_n {M}_n^2.
\end{align}
Adding~\eqref{eq1} and~\eqref{eq2}, we get the desired equation.
\end{proof}

Now, we will derive a second order differential operator for the sequence of matrix function \( \{Y_n \}_{n \in \mathbb N} \) associated with a generalized semiclassical weight.

\begin{proposition}
\label{propeqdif2}
Let \(\nu \) be a
generalized
semiclassical weight on \( \mathbb T \).
Then \(\{ Y_n \}_{n \in \mathbb N} \)
 satisfies 
\begin{align}\label{op2orden}
 Y_n^{\prime\prime}+ 2Y_n^{\prime} C_n^{\prime}C_n^{-1}+Y_n C_n^{\prime\prime} C_n^{-1} = \big( {M}_n^{\prime}+ {M}_n^2 \big) \, Y_n.
\end{align}
\end{proposition}

\begin{proof}
Taking derivative on~\eqref{Mn}, i.e. \( Z_n^{\prime}Z_n^{-1} ={M}_n \) we arrive to
\begin{align*}
Z_n^{\prime\prime} Z_n^{-1} - \big( Z_n^{\prime}Z_n^{-1} \big) \big( Z_n^{\prime}Z_n^{-1} \big) = {M}_n^{\prime} .
\end{align*}
Again, using~\eqref{Mn}, we get
\begin{align}\label{auxiliardadem}
Z_n^{\prime\prime} = \big( {M}_n^{\prime} + M_n^2 \big) Z_n .
\end{align}
By definition, \( Z_n=Y_nC_n \), and so
\begin{align*}
Z_n^{\prime}& =Y_n^{\prime}C_n+ Y_n^{\prime}C_n^{\prime} ,
 &&
Z_n^{\prime\prime} =Y_n^{\prime\prime}C_n+2Y_n^{\prime}C_n^{\prime}+Y_nC_n^{\prime\prime} .
\end{align*}
Ploughing this into~\eqref{auxiliardadem} we get the desired result.
\end{proof}

The second order differential equation in Proposition~\ref{propeqdif2}, given for \( \{Y_n\}_{n \in \mathbb N} \), could trace back to a first order one for the \( \{Z_n\}_{n \in \mathbb N} \), i.e.~\eqref{Mn}.

\begin{proposition}\label{novaprop} 
Let \( \nu \) be a H\"older type weight defined on \( \mathbb T \).
If \(\{ Y_n \}_{n \in \mathbb N} \) satisfies~\eqref{op2orden}, then~\( \{ Z_n \}_{n \in \mathbb N} \) satisfies
\begin{align}
\label{eq:semic}
Z_n^\prime = -T_n^{-1} (-z) \Big\{ z \big( (M_{n+1}^\prime + M_{n+1}^2)T_n - T_n (M_{n}^\prime + M_{n}^2) \big) + T_n^\prime - \frac 3 {4z} T_n \Big\} Z_n , && n \in \mathbb N .
\end{align}
\end{proposition}

\begin{proof}
Multiplying equation~\eqref{op2orden} by \( C_n \), and taking into account~\eqref{Zn} we get that~\eqref{auxiliardadem} takes place. Rewriting \eqref{auxiliardadem} in \( n+1 \) we see that
\begin{align*}
Z_{n+1}^{\prime\prime} = (M_{n+1}^\prime + M_{n+1}^2)Z_{n+1} .
\end{align*}
By Proposition~\ref{propnewZn} we know that \( Z_{n+1} = z^{-1/2} T_n Z_n \), hence substituting this into the last equation, and making some simplifications we arrive to
\begin{align*}
\frac 3 4 T_n Z_n - z (T_n^\prime Z_n + T_n Z_n^\prime) + z^2 \big( 2 T_n^\prime Z_n^\prime + T_n Z_n^{\prime\prime} \big) = z^2 (M_{n+1}^\prime + M_{n+1}^2) T_n Z_n .
\end{align*}
Now, applying our hypothesis, i.e. \( Z_{n}^{\prime\prime} = (M_{n}^\prime + M_{n}^2) Z_n \) we get
\begin{align*}
(-T_n + 2 z T_n^\prime) Z_n^\prime = \Big\{ z \big( (M_{n+1}^\prime + M_{n+1}^2)T_n - T_n (M_{n}^\prime + M_{n}^2) \big) + T_n^\prime - \frac 3 {4z} T_n \Big\} Z_n ,
\end{align*}
and taking into account that \( -T_n (z) + 2 z T_n^\prime (z) = - T_n (-z) \)
we arrive to~\eqref{eq:semic}.
\end{proof}


\section{Examples}\label{sec4}

As we have seen in Section~\ref{sec3} the matrix functions \( M_n \) are a key to understanding generalized semiclassical orthogonal polynomials. When \( A \) is equal to \( z \) or \( z-1 \) the weight function is of the type
\begin{align*}
 \nu (z) = (1-z)^\alpha z^\beta \dde^{-\gamma/z} H (z) ,
\end{align*} 
for some constants \( \alpha, \beta, \gamma \in \mathbb C \)
and where \( H \) is an entire function.

Here we will take the generalized modified Bessel and generalized modified Jacobi as a case study.

\subsection{Modified Bessel}

Here we consider the generalized semiclassical weight function 
\begin{align*}
w(\theta ) = 
 \dde ^{\ell \cos(\theta )} H ( \dde^{\ddi \theta})&& \mbox{and} &&
\nu(z) = 
\dde^{\ell (z+z^{-1})/2} H (z), && z=\dde^{\ddi \theta}, && \theta \in[0,2\pi],
\end{align*}
\noindent where \( \ell>0 \) is a real parameter. In the case when \( H (z) = 1 \), the Verblunsky coefficients are real and depending of \( \ell \). This is the weight function related to the modified Bessel polynomials~\cite{Is05}.

We know from Proposition~\ref{teo48} that the matrix \( Z_n \) defined in~\eqref{Zn} related to the generalized modified Bessel weight function is analytic in \( \mathbb C \setminus (\mathbb T \cup \{ 0 \}) \), as \( \mathcal Z_A = \{ 0 \} \).
Furthermore, from Proposition~\ref{teo:geral}, we know that the corresponding structure matrix~\( M_n \), \( n\in\mathbb{N} \), given in~\eqref{Mn}, is a holomorphic function in \( \mathbb {C} \setminus \mathcal{Z}_A \).

Now we will study the analytic character for the structure matrix associated with orthogonal polynomials in relation to the \emph{generalized modified Bessel} weight function.

\begin{proposition}
\label{teo:mnbessel}
Let \( \nu(z)=\dde ^{\ell(z+z^{-1})/2} H(z) \), where \( H \) is an entire function.
Then, the corresponding structure matrix \( M_n \), \( n\in\mathbb{N} \), given in~\eqref{Mn}, is a holomorphic function in \( \mathbb {C}\setminus \{0\}
 \) with a pole of order \( 2 \) at \( z=0 \).
\end{proposition}

\begin{proof}
Note~that
\begin{align*}
M_n(z) &= Z_n^\prime (z)Z_n^{-1}(z) = Y_n^\prime (z)Y_n^{-1}(z)+Y_n(z)C_n^\prime(z)C_n^{-1}(z)Y_n^{-1}(z). 
\end{align*}
Multiplying both sides of the last equation by \(z^2 \), we get 
\begin{align*}
z^2M_n(z)=z^2Y_n^\prime(z)Y_n^{-1}(z)+z^2Y_n(z)C_n^\prime(z)C_n^{-1}(z)Y_n^{-1}(z).
\end{align*}
Therefore, using~\eqref{G_n_0},
we obtain 
\begin{align*}
& \lim_{z \to 0}z^2M_n(z)={\begin{bmatrix}
\frac{ \ell }{4}(b_{n-1}b_n^{-1}-\alpha_{n-1}^2) & -\frac{ \ell }{2}b_n^{-1}\alpha_{n-1}\\
 -\frac{ \ell }{2}b_{n-1}\alpha_{n-1} & -\frac{ \ell }{4}( b_{n-1}b_n^{-1}-\alpha_{n-1}^2)
\end{bmatrix}}\neq \mathbf{0}_{2},
\end{align*}
where \( \mathbf{0}_{2} \) is the zero matrix.
It follows that \( M_n \) has a pole of order two at \( z=0 \).
\end{proof}

In the hypothesis of the Proposition~\ref{teo:mnbessel} we have that
\( \widetilde{M}_n=z^2 M_n \), 
is, for each \( n \in \mathbb N \), an entire matrix function.
Now, we will explicitly determine the \( \widetilde M_n \) just defined for the case when \( H (z) = 1\).

%
We know that \( \widetilde{M}_n(z)=z^2Y_n^\prime(z)Y_n^{-1}(z)+z^2Y_n(z)C_n^\prime(z)C_n^{-1}(z)Y_n^{-1}(z) \).
 From \eqref{ynnew},~\eqref{inversaynnew}, and~\eqref{Zn}, and using~\eqref{Gninf} and~\eqref{Gnestrelainf}, we obtain the asymptotic behavior for~\( z\to\infty \),
{\small
\begin{align*}
 \widetilde{M}_n(z)
 =
\begin{bmatrix}
\frac{\ell \, z^2}{4}+\frac{n z}{2}-\frac{b_{n-1}b_n^{-1}\alpha_n\alpha_{n-2}+ \ell +4\Phi_1^n}{4} & 
\frac{\alpha_n}{b_{n}} ( \frac{\ell z}{2}+n+1+\frac{ \ell }{2}(\Phi_1^n-1))+\frac{ \ell \alpha_{n-1}}{2b_{n+1}}\\
b_{n-1}\big( ( \frac{\ell z}{2} + n-1-\frac{ \ell }{2}\Phi_1^n)\alpha_{n-2}+\overline{\Phi_{n-2}^{n-1}} \big) & 
-\frac{\ell \, z^2}{4}-\frac{n z}{2}-\frac{b_{n-1}b_n^{-1}\alpha_n\alpha_{n-2}- \ell -4\Phi_1^n}{4} 
\end{bmatrix}
+ \operatorname{O}(z^{-1}) .
\end{align*}
}
{\hspace{-.075cm}}Therefore, since \( \widetilde{M}_n \) is an entire matrix function, the Liouville Theorem implies that
\begin{align}\label{segundaparteMntilde}
 \widetilde{M}_n(z)=
 \begin{bsmallmatrix}
\frac{ \ell }{4}z^2+\frac{n}{2}z-\frac{1}{4}(b_{n-1}b_n^{-1}\alpha_n\alpha_{n-2}+\ell +4\Phi_1^n) 
& 
b_{n}^{-1}\frac{ \ell }{2}\alpha_n z+b_{n}^{-1}\alpha_n ( n+1+\frac{ \ell }{2}(\Phi_1^n -1) ) +\frac{ \ell }{2}b_{n+1}^{-1}\alpha_{n-1}
 \\
b_{n-1}\big( \frac{ \ell }{2}\alpha_{n-2}z+b_{n-1} ( n-1-\frac{ \ell }{2}\Phi_1^n)\alpha_{n-2}+\overline{\Phi_{n-2}^{n-1}} \big) 
&
 -\frac{ \ell }{4}z^2-\frac{n}{2}z-\frac{1}{4}(b_{n-1}b_n^{-1}\alpha_n\alpha_{n-2}- \ell -4\Phi_1^n) 
\end{bsmallmatrix} 
 .
\end{align}
Note that \( \widetilde{M}_n \) in \eqref{segundaparteMntilde} is a polynomial function of degree \( 2 \), that is, we can write it as \( \widetilde{M}_n(z)=\widetilde{F}_n^2 z^2+\widetilde{F}_n^{1}z+\widetilde {F}_n^{0} \), where \( \widetilde{F}_n^{2} \), \( \widetilde{F}_n^{1} \), and \( \widetilde{F}_n^0 \) are constants and
\(\displaystyle
\lim_{z \to 0}\widetilde{M}_n(z)=\widetilde{F}_n^{0} \).
Therefore, using the 
Lemma \ref{lemma6}, we get
\begin{align*}
& \widetilde{F}_n^{0}={\begin{bmatrix}
\frac{ \ell }{4} ( b_{n-1}b_n^{-1}-\alpha_{n-1}^2 ) & -\frac{ \ell }{2}b_n^{-1}\alpha_{n-1}\\
 -\frac{ \ell }{2}b_{n-1}\alpha_{n-1} & -\frac{ \ell }{4} ( b_{n-1}b_n^{-1}-\alpha_{n-1}^2)
\end{bmatrix}}.
\end{align*}
Moreover, from \eqref{segundaparteMntilde}, we have
\begin{align*}
& \widetilde{F}_n^{1}={\begin{bmatrix}
\frac{n}{2} & \frac{ \ell }{2}b_{n}^{-1}\alpha_n \\
\frac{ \ell }{2}b_{n-1}\alpha_{n-2} & -\frac{n}{2}
\end{bmatrix}} &&
\mbox{and}
&& \widetilde{F}_n^{2}={\begin{bmatrix}
\frac{ \ell }{4} & 0 \\
0 & -\frac{ \ell }{4}
\end{bmatrix}}.
\end{align*}
Therefore, when \( z\to 0 \), we arrive 
to
\begin{align}\label{matriztilde}
& \widetilde{M}_n(z)=
\begin{bmatrix}
\frac{ \ell }{4}z^2+\frac{n}{2}z+\frac{ \ell }{4} ( b_{n-1}b_n^{-1}-\alpha_{n-1}^2) & -\frac{ \ell }{2}b_n^{-1}(\alpha_{n-1}-\alpha_n z) \\ 
-\frac{ \ell }{2}b_{n-1}(\alpha_{n-1}-\alpha_{n-2}z) & -\frac{ \ell }{4}z^2-\frac{n}{2}z-\frac{ \ell }{4} ( b_{n-1}b_n^{-1}-\alpha_{n-1}^2 )
\end{bmatrix} .
\end{align}

The zero curvature formula is important for finding nonlinear difference equations satisfied by Verblunsky coefficients.
As a consequence of Proposition~\ref{propnula} we can see that
the matrix function, \( \widetilde{M}_n \), defined in \eqref{matriztilde} satisfies
\begin{align}\label{curvatureformulabessel}
z^2T_n^\prime +T_n\widetilde{M}_n-\frac{z}{2}T_n-\widetilde{M}_{n+1}T_n=
\mathbf{0}_{2} ,
&& n \in \mathbb N .
\end{align}
In fact,
we only have to multiply equation~\eqref{curvatureformula} by \( z^2 \) and using the fact that \( \widetilde M_n = z^2 M_n \), \( n \in \mathbb N \).
As an easy consequence of~\eqref{curvatureformulabessel} we get that
\begin{align}\label{dp2curvature}
\alpha_{n}(\ell) + \alpha_{n-2}(\ell) = -\frac{2n}{ \ell } \frac{\alpha_{n-1}(\ell)}{1-\alpha_{n-1}^2(\ell)}, && n \geqslant 2,
&& \ell \in \mathbb R .
\end{align}
The equation \eqref{dp2curvature} corresponds to the discrete Painlev\'e equation (dP\(_{\mbox{II}} \)).

%

Now, we derive the differential properties for the functions \( \Phi_n \), \( \Phi_n^* \), \( G_n \) and \( G_n^* \) coming from the Riemann--Hilbert problem associated with the modified Bessel weight.
In fact, multiplying by \( z^2 \), equation~\eqref{Mn} we obtain
\begin{align*}
z^2 Z_n^\prime(z) = \widetilde M_n (z) Z_n (z), && n \in \mathbb N .
\end{align*}
Taking into account~\eqref{Zn} we arrive to
\begin{align}\label{op1bessel}
 z^2Y_n^\prime 
 =
\widetilde{M}_nY_n-z^2Y_nC_n^\prime C_n^{-1} , && n \in \mathbb N . 
\end{align}
The first order matrix differential equation \eqref{op1bessel} splits into the following differential relations 
\begin{align}
 z^2\Phi_n^\prime (z) & =\bigg\{nz+\frac{ \ell }{2}-\frac{ \ell }{2}\alpha_{n-1}^2\bigg\}\Phi_n(z)+\frac{ \ell }{2}b_{n-1}b_n^{-1} ( \alpha_{n-1}-\alpha_n z ) \Phi_{n-1}^*(z)
 \label{eqdiferbessel}
 ,
 \\
 \nonumber z^2G_n^\prime (z) & =\bigg\{\frac{ \ell }{2}z^2-\frac{ \ell }{2}\alpha_{n-1}^2\bigg\}G_n(z)+\frac{ \ell }{2}b_{n-1}b_n^{-1} ( \alpha_{n-1}-\alpha_n z ) G_{n-1}^*(z)
 ,
 \\
\nonumber z^2 \big( \Phi_{n-1}^* \big)^\prime (z) & 
 =\frac{ \ell }{2} ( \alpha_{n-1}-\alpha_{n-2}z ) \Phi_n(z)+\bigg\{-\frac{ \ell }{2}z^2+\frac{ \ell }{2}\alpha_{n-1}^2\bigg\} \Phi_{n-1}^*(z)
 ,
 \\
\nonumber z^2 \big( G_{n-1}^* \big) ^\prime (z) & 
 =\frac{ \ell }{2} ( \alpha_{n-1}-\alpha_{n-2}z ) G_n(z)+\bigg\{-nz-\frac{ \ell }{2}+\frac{ \ell }{2}\alpha_{n-1}^2\bigg\}G_{n-1}^*(z)
 .
 \end{align} 


Note that with some manipulations, it can be shown that equation \eqref{eqdiferbessel} is equivalent to the structure relation
\begin{align}\label{relest}
z\Phi_{n}^\prime(z)=n\Phi_n(z)+\frac{\ell}{2}\frac{\kappa_{n-1}^2}{\kappa_n^2}(\Phi_{n-1}(z)-\overline{\alpha}_n\Phi_{n-1}^*(z))
\end{align}
presented in \cite{BRS23}. 
Furthermore, doing \(z=0 \) in \eqref{relest} and some more calculations, 
it can be shown that \eqref{eqdiferbessel} coincides with equation \eqref{structurebessel}. 

Applying Proposition~\ref{propeqdif2}, we can derive second order differential equation for the matrix~\(Y_n \). In fact, we only have to multiply~\eqref{op2orden} by \( z^2 \), to obtain
\begin{align*}
z^2Y_n^{\prime\prime}+2Y_n^{\prime} ( z^2C_n^{\prime}C_n^{-1} )+Y_n ( z^2 C_n^{\prime\prime} C_n^{-1}) =\Big( z^2 {M}_n^{\prime} + \frac{\widetilde{M}_n^2}{z^2} \Big) \, Y_n.
\end{align*}
We can see that
\( \displaystyle z^2 {M}_n^{\prime} = \widetilde M_n^\prime - \frac{2}{z} \widetilde M_n \)
and
\begin{align*}
 \widetilde M_n = z^2 Z_n^\prime Z_n^{-1} =
 z^2 \big( Y_n^\prime Y_n^{-1} + Y_n C_n^\prime C_n^{-1} Y_n^{-1} \big).
\end{align*}
Substituting this into the last equation we arrive to
\begin{align}\label{op2bessel}
 z^2Y_n^{\prime\prime}+2Y_n^{\prime} ( z^2C_n^{\prime}C_n^{-1}+z \, \mathbf{\operatorname I}_2 )+Y_n ( z^2C_n^{\prime\prime} C_n^{-1}+2z \, C_n^{\prime} C_n^{-1} ) =\Big( \widetilde{M}_n^{\prime}+\frac{\widetilde{M}_n^2}{z^2} \Big) \, Y_n.
\end{align}

The second order matrix differential equation \eqref{op2bessel} splits into the following differential~relations 
\begin{align*}
\begin{multlined}[t][.95\textwidth]
z^2\Phi_n^{\prime\prime}(z)+\left( \frac{ \ell }{2}z^2+(2-n)z-\frac{ \ell }{2}\right) \Phi_n^{\prime}(z)+\bigg( -\frac{\ell n}{2}z-\frac{\ell^2}{4}-n\\
 -\frac{\ell^2}{4}((1-\alpha_{n-1}^2)\alpha_n\alpha_{n-2}-\alpha_{n-1}^2) \bigg) \Phi_n(z)=-\frac{ \ell }{2}(1-\alpha_{n-1}^2)\alpha_n\Phi_{n-1}^*(z) ,
\end{multlined} 
 \\
\begin{multlined}[t][.95\textwidth]
z^2G_n^{\prime\prime}(z)+\left( -\frac{\ell z^2}{2}+\left(n+2\right) z+\frac{ \ell }{2}\right) G_n^{\prime}(z)-\bigg( \ell\big(\frac{n}{2}+1\big) z+\frac{\ell^2}{4}\\
+\frac{\ell^2}{4}( (1-\alpha_{n-1}^2)\alpha_n\alpha_{n-2}-\alpha_{n-1}^2]\bigg) G_n(z)=-\frac{ \ell }{2}(1-\alpha_{n-1}^2)\alpha_n G_{n-1}^*(z),
\end{multlined}
 \\ 
\begin{multlined}[t][.95\textwidth]
z^2 \big( \Phi_{n-1}^* \big) ^{\prime\prime}(z)
+\left( \frac{\ell z^2}{2}+\left(2-n\right) z -\frac{ \ell }{2}\right) \big( \Phi_{n-1}^* \big)^{\prime}(z)-\bigg( \ell \big(\frac{n}{2}-1\big)z+\frac{\ell^2}{4} \\
+\frac{\ell^2}{4} ( (1-\alpha_{n-1}^2)\alpha_n\alpha_{n-2}-\alpha_{n-1}^2) \bigg) \Phi_{n-1}^*(z)=-\frac{ \ell }{2}\alpha_{n-2}\Phi_n(z),
\end{multlined}
\\ 
\begin{multlined}[t][.95\textwidth]
z^2 \big( G_{n-1}^* \big)^{\prime\prime}(z)+\left(-\frac{\ell z^2}{2}+\left(n+2\right) z+\frac{ \ell }{2}\right) \big( G_{n-1}^* \big)^{\prime}(z)-\bigg( \frac{\ell n}{2}z+\frac{\ell^2}{4}-n
 \\
+\frac{\ell^2}{4}( (1-\alpha_{n-1}^2)\alpha
_n\alpha_{n-2}-\alpha_{n-1}^2) \bigg)G_{n-1}^*(z)=-\frac{ \ell }{2}\alpha_{n-2}G_n(z).
\end{multlined}
\end{align*}

To get these second order differential relations we only have to substitute the representation of \(\widetilde M_n \) given in~\eqref{matriztilde} into~\eqref{op2bessel}.
In fact,
from~\eqref{matriztilde},
we see that
\begin{align*}
\frac{
\widetilde M_n^2 
}{z^2 } 
 & =
 \Big(\frac{\ell^2}{16}z^2+\frac{\ell n}{4}z+\frac{\ell^2}{8}+\frac{n^2}{4}-\frac{\ell^2}{4}\alpha_{n-1}^2+\frac{\ell^2}{4}(1-\alpha_{n-1}^2)\alpha_{n-2}\alpha_n+\frac{\ell n}{4z}+\frac{\ell^2}{16z^2}\Big)
 \, \mathbf{\operatorname I}_2,\\
\widetilde M_n^\prime & =
\begin{bmatrix}
 \frac{\ell}{2}z+\frac{n}{2} & \frac{\ell}{2}b_n^{-1}\alpha_n \\
 \frac{\ell}{2}b_{n-1}\alpha_{n-2} & -\frac{\ell}{2}z-\frac{n}{2} \\
\end{bmatrix}.
\end{align*}
Furthermore, taking into account the definition of \( C_n \) in \eqref{Zn},
we have
\begin{align*}
z^2C_n^{\prime\prime}C_n^{-1}+2zC_n^{\prime}C_n^{-1}
& 
= \big( \frac{\ell^2}{16}z^2-\frac{\ell n}{4} z-\frac{\ell^2}{8}+\frac{n^2}{4}+\frac{\ell n}{4z}+\frac{\ell^2}{16z^2} \big) \mathbf{\operatorname I}_2 +
 \begin{bmatrix}
\frac{\ell}{2} z-\frac{n}{2} & 0\\
0 & - \frac{\ell}{2} z + \frac{n}{2}
\end{bmatrix} , \\
z^2C_n^{\prime}C_n^{-1}+z\mathbf{\operatorname I}_2
 &=
\begin{bmatrix}
\frac{\ell}{4}z^2+\big(1-\frac{n}{2}\big)z-\frac{\ell}{4} & 0\\
0 & -\frac{\ell}{4}z^2+\big(1+\frac{n}{2}\big)z+\frac{\ell}{4}
\end{bmatrix} .
\end{align*}
By substituting all these matrices into \eqref{op2bessel}, we obtain the desired equations.

\subsection{Jacobi modified}
We consider here a generalization of the semiclassical weight function, studied in~\cite{Ra10}, namely
\begin{align*}
w(\theta )=\tau(b)\dde ^{-\eta \theta} \big( \sin^2(\theta /2) \big)^{\lambda} H (\dde^{ \ddi \theta}), && \theta \in [0,2\pi],
\end{align*}
where \( H \) is an entire function, and
\begin{align*}
\eta \in \mathbb{R}, &&
\lambda > -1/2 , &&
b=\lambda + \ddi \eta,
&& \text{and} &&
 \tau(b)= \frac{\dde ^{\pi\eta} 2^{b+\overline{b}}|\Gamma(b+1)|^2}{2\pi\Gamma(b+\overline{b}+1)}.
\end{align*}
Since \( z=\dde ^{\ddi \theta} \), we can write the weight function as
\begin{align*}
\nu(z)=\frac{\tau(b)}{2^{b+\overline{b}}}(-z)^{-\overline{b}}(1-z)^{b+\overline{b}} H (z), && z\in\mathbb{T}.
\end{align*}
For simplicity, in this work we consider
\( \nu(z)=(-z)^{-\overline{b}}(1-z)^{b+\overline{b}} H (z) \),
\(z\in\mathbb{T}\),
which coincides up to a constant with the above weight.


%

We will study the analytic character for the structure matrix associated with these \emph{generalized modified Jacobi} orthogonal polynomials.

\begin{proposition}
\label{teo:mnjacobi}
Let \(
\nu(z)=(-z)^{-\overline{b}}(1-z)^{b+\overline{b}} H (z)\), then the structure matrix \( M_n \),
 \( n\in\mathbb{N} \), is~a holomorphic function in \( \mathbb{C}\setminus
 \{0,1\}
 \) with a simple pole at \( z=0 \) and a removable singularity or a simple pole at \( z=1 \).
\end{proposition}

\begin{proof} 
From Proposition~\ref{teo:geral}, it follows that \( M_n \) has isolated singularities at \( z=0 \) and \( z=1 \). 
Multiplying \( M_n \) by \( z \) we have
\begin{align*}
zM_n=zY_n^\prime Y_n^{-1}+zY_nC^\prime_n C^{-1}_n Y_n^{-1}.
\end{align*}
Note that, using~\eqref{G_n_0}, we get
\begin{align*}
 \lim_{z \to 0}zM_n(z)
 & =
 \begin{bmatrix}
-\frac{1}{2}(\overline{b}+n)(2|\alpha_{n-1}|^2-1) & -b_n^{-1}(\overline{b}+n)\overline{\alpha}_{n-1} \\
 -b_{n-1}(\overline{b}+n)\alpha_{n-1} & \frac{1}{2}(\overline{b}+n)(2|\alpha_{n-1}|^2-1)
\end{bmatrix}
 \neq \mathbf{0}_{2} .
\end{align*}
Therefore, by direct calculation, we can state that \( M_n \) has a simple pole at \( z=0 \).

To prove the statement related to the singularity at the point \( z=1 \), it is sufficient to prove that
\( \displaystyle
\lim_{z \to 1} (z-1)^2M(z)=\mathbf{0}_{2} \).
At the point \( z=1 \), using~\cite[Ch. 8.3]{gakhov} and~\cite[Lemma~7.2.2]{AF97}, we can state that
\begin{align*}
G_n(z) = 
\begin{cases}
\operatorname{O}(1) , & b+ \overline{b} = \beta + \ddi \gamma, \quad \beta > 0, \\
\operatorname{O}\big( \log (z-1) \big) , & b+ \overline{b} = \ddi \gamma, \\
\operatorname{O}\big( (z-1)^{\beta} \big) , & b+ \overline{b} = \beta + \ddi \gamma, \quad -1 <\beta < 0 ,
\end{cases}
\end{align*}
and analogously
\begin{align*}
G^*_n(z) =
\begin{cases}
\operatorname{O}(1) , & b+ \overline{b} = \beta + \ddi \gamma, \quad \beta > 0, \\
\operatorname{O}\big( \log (z-1) \big) , & b+ \overline{b} = \ddi \gamma, \\
\operatorname{O}\big( (z-1)^{\beta} \big) , & b+ \overline{b} = \beta + \ddi \gamma, \quad -1 <\beta < 0 .
\end{cases}
\end{align*}
It also follows that
\(\displaystyle
G^{\prime}_n(z) = \frac{1}{2\pi \ddi }\int_\mathbb{T}
\frac{\Phi_n(t)}{(t-z)^2}\frac{\nu(t)}{t^n}
\dd t \).
So, taking into account that
\begin{align*}
(z-1)G^{\prime}_n(z) & = \frac{1}{2\pi \ddi }\int_\mathbb{T}
\frac{(z-1)\Phi_n(t)}{(t-z)^2}\frac{\nu(t)}{t^n} \dd t \\
& = -\frac{1}{2\pi \ddi }\int_\mathbb{T}
\frac{\Phi_n(t)}{t-z}\frac{\nu(t)}{t^n}
\dd t + \frac{1}{2\pi \ddi }\int_\mathbb{T}
\frac{(t-1)\Phi_n(t)}{(t-z)^2}\frac{\nu(t)}{t^n}
\dd t ,
\end{align*}
and using integration by parts, and again the behavior of the Cauchy integral at the power singularity \((1-z)^{b + \overline{b}} \), we get that
 \begin{align*}
(z-1)G^{\prime}_n(z) = 
\begin{cases}
\operatorname{O}(1) , & b+ \overline{b} = \beta + \ddi \gamma, \quad \beta > 0, \\
\operatorname{O}\big( \log (z-1) \big) , & b+ \overline{b} = \ddi \gamma, \\
\operatorname{O}\big( (z-1)^{\beta} \big) , & b+ \overline{b} = \beta + \ddi \gamma, \quad -1 <\beta < 0 .
\end{cases}
\end{align*}
Analogously,
\begin{align*}
(z-1){G^*_n}^{\prime}(z) =
\begin{cases}
\operatorname{O}(1) , & b+ \overline{b} = \beta + \ddi \gamma, \quad \beta > 0 , \\
\operatorname{O} \big( \log (z-1) \big) , & b+ \overline{b} = \ddi \gamma, \\
\operatorname{O}\big( (z-1)^{\beta} \big) , & b+ \overline{b} = \alpha + \ddi \gamma, \quad -1 <\beta < 0 .
\end{cases}
\end{align*}
So, it holds in any case that
\begin{align*}
\lim_{z\to 1} (z-1) G_n(z) & = 0, && \lim_{z\to 1} (z-1) G^*_n(z) = 0, \\
\lim_{z\to 1} (z-1)^2 G^{\prime}_n(z) & = 0, &&
\lim_{z\to 1} (z-1)^2 {G^*}^{\prime}_n(z) = 0 ,
\end{align*}
and, using that \(\det Y_n = 1 \), we arrive to
\begin{align*}
Y_{n} ^{\prime}(z)
=
\begin{bmatrix}
\operatorname{O}(1) & \operatorname{o}(\frac{1}{(z-1)^2}) \\
\operatorname{O}(1) & \operatorname{o}(\frac{1}{(z-1)^2})
\end{bmatrix}, 
&&
Y_{n}^{-1}(z) 
=
\begin{bmatrix}
\operatorname{o}\big( (z-1)^{-1} \big) & \operatorname{o} \big( (z-1)^{-1} \big) \\
\operatorname{O}(1) & \operatorname{O}(1)
\end{bmatrix}, 
&& z \to 1.
\end{align*}
From these we successively get that
{\small
\begin{align*} 
 & \lim _{z \to 1} (z-1)^{2}\left(Y_{n}\right)^{\prime}\left(Y_{n}\right)^{-1} \\
 & \phantom{ola} =
\lim _{z \to 1} (z-1)^{2}
\begin{bmatrix}
 \operatorname{o} \big( (z-1)^{-1} \big)+ \operatorname{o} \big( (z-1)^{-2} \big) 
 &\operatorname{o} \big( (z-1)^{-1} \big)+\operatorname{o} \big( (z-1)^{-2} \big) \\
 \operatorname{o} \big( (z-1)^{-2} \big) +\operatorname{o} \big( (z-1)^{-1} \big) 
 & \operatorname{o} \big( (z-1)^{-2} \big) + \operatorname{o} \big( (z-1)^{-1} \big)
\end{bmatrix}
 \\
& =\lim_{z \to 1} (z-1)^{2}
\begin{bmatrix}
 \operatorname{o} \big( (z-1)^{-2} \big) &
 \operatorname{o} \big( (z-1)^{-2} \big) \\
\operatorname{o} \big( (z-1)^{-2} \big) &
\operatorname{o}\big( (z-1)^{-2}\big)
\end{bmatrix}
= \mathbf{0}_{2} ,
\end{align*}}
{\hspace{-.075cm}}and
{\small
\begin{align*}
& \lim_{z \to 1} (z-1)^2 Y_n C^\prime_n C^{-1}_n Y_n^{-1} \\
& 
\phantom{ola} = 
\lim _{z \to 1} (z-1)^2 
\begin{bsmallmatrix}
\operatorname{O}(1) & \operatorname{o} ( (z-1)^{-1} ) \\
\operatorname{O}(1) & \operatorname{o} ( (z-1)^{-1} )
\end{bsmallmatrix}
\begin{bsmallmatrix}
-\frac{ n+\overline{b} }{2 z} + \frac{ b+\overline{b} }{2 (z-1)} + \operatorname O (1) & 0 \\
0& \frac{n+\overline{b}}{2 z} - \frac{ b+\overline{b} }{2 (z-1)} + \operatorname O (1)
\end{bsmallmatrix}
\begin{bsmallmatrix}
\operatorname{o} ( (z-1)^{-1} ) & \operatorname{o} ( (z-1)^{-1} ) \\
\operatorname{O}(1) & \operatorname{O}(1) 
\end{bsmallmatrix}
 \\
& \phantom{ola} = \lim _{z \to 1} (z-1)^2
\begin{bsmallmatrix}
\operatorname{O} ( (z-1)^{-1} ) 
\operatorname{o} ( (z-1)^{-1} ) &
 \operatorname{O} ( (z-1)^{-1} ) 
 \operatorname{o} ( (z-1)^{-1} ) \\
\operatorname{O} ( (z-1)^{-1} )
 \operatorname{o} ( (z-1)^{-1} ) & 
\operatorname{O} ( (z-1)^{-1} ) 
\operatorname{o} ( (z-1)^{-1} )
\end{bsmallmatrix}
 =
\mathbf{0}_{2} .
\end{align*}}
{\hspace{-.1cm}}By combining previous results, we obtain
\(\displaystyle
\lim_{z \to 1} (z-1)^{2} M_{n}= \mathbf{0}_{2}
 \)
as we wanted to prove.
\end{proof}

In the hypothesis of the Proposition~\ref{teo:mnjacobi}, we have that
\( \widetilde{M}_n=z(1-z)M_n \)
is an entire matrix function.
Now, we will explicitly determine the \( \widetilde M_n \) just defined when \( H (z) = 1 \).
We know that
\begin{align*}
z(1-z)M_n=z(1-z)Y_n^\prime Y_n^{-1}+z(1-z)Y_nC^\prime_n C^{-1}_n Y_n^{-1}.
\end{align*}
From \eqref{ynnew}, \eqref{inversaynnew}, and \eqref{Zn}, then by \eqref{Gninf} and \eqref{Gnestrelainf}, we obtain the following asymptotic behavior about infinity, 
{\small
\begin{align*}
& \widetilde{M}_n(z)
={\begin{bmatrix}
\displaystyle
-\frac{(b+n)z+(\overline{b}-n)}{2}+\Phi_1^n
& -(b+n+1)b_n^{-1}\overline{\alpha}_n \\
-(b+n-1)b_{n-1}\alpha_{n-2} & 
\displaystyle
\frac{(b+n)z+(\overline{b}-n)}{2} - \Phi_1^n
\end{bmatrix}} +\operatorname{O}(z^{-1}) 
 .
\end{align*}
}
{\hspace{-.075cm}}Using the Liouville Theorem, it follows that
\begin{align}\label{matriz1}
& \widetilde{M}_n(z)
=
\begin{bmatrix}
-\frac{1}{2}( (b+n)z+(\overline{b}-n) )+\Phi_1^n & -(b+n+1)b_n^{-1}\overline{\alpha}_n \\
-(b+n-1)b_{n-1}\alpha_{n-2} & \frac{1}{2} ( (b+n)z+(\overline{b}-n) ) -\Phi_1^n
\end{bmatrix} 
 .
\end{align}
Note that \( \widetilde{M}_n(z)=\widetilde{F}_n^{1}z+\widetilde{F}_n^{0} \), where \( \widetilde{F}_n^{1} \), and \( \widetilde{F}_n^{0} \) are constants and \( \displaystyle \lim_{z \to 0}z(1-z)M_n=\widetilde{F}_n^{0} \).
 Therefore, applying the Lemma \ref{lemma6}, we get
\begin{align*}
& \widetilde{F}_n^{0}={\begin{bmatrix}
-\frac{1}{2}(\overline{b}+n) ( 2|\alpha_{n-1}|^2-1 ) & -b_n^{-1}(\overline{b}+n)\overline{\alpha}_{n-1} \\
 -b_{n-1}(\overline{b}+n)\alpha_{n-1} & \frac{1}{2}(\overline{b}+n) ( 2|\alpha_{n-1}|^2-1) 
\end{bmatrix}} .
\end{align*}
Moreover, from \( \widetilde{M}_n \) in \eqref{matriz1}, we have
\begin{align*}
& \widetilde{F}_n^{1}={\begin{bmatrix}
-\frac{1}{2}(b+n) & 0 \\
 0 & \frac{1}{2}(b+n)
\end{bmatrix}},
\end{align*}
and so we arrive to
{\small
\begin{align}\label{matrizhat}
\widetilde{M}_n(z)
 =
\begin{bmatrix}
\displaystyle
-\frac{(b+n)z+(\overline{b}+n)\big(2|\alpha_{n-1}|^2-1\big)}{2} & -b_{n}^{-1}(\overline{b}+n)\overline{\alpha}_{n-1} \\
-b_{n-1} (\overline{b}+n)\alpha_{n-1} &
\displaystyle
\frac{(b+n)z+(\overline{b}+n)\big(2|\alpha_{n-1}|^2-1\big)}{2}
\end{bmatrix}.
\end{align}
}
{\hspace{-.075cm}}It is easy to see that the residue matrix of \( M_n \) at the pole \( z=1 \) is given by
{\small
\begin{align*}
 \lim_{z \to 1}(z-1)M_n(z)
 & =
\begin{bmatrix}
\displaystyle
\frac{(b+n)+(\overline{b}+n)\big(2|\alpha_{n-1}|^2-1\big)}{2} & b_{n}^{-1}(\overline{b}+n)\overline{\alpha}_{n-1} \\
b_{n-1} (\overline{b}+n)\alpha_{n-1} &
\displaystyle
-\frac{(b+n)+(\overline{b}+n)\big(2|\alpha_{n-1}|^2-1\big)}{2}
\end{bmatrix}.
\end{align*}}
{\hspace{-.075cm}}Moreover,
comparing the \((1,1) \) and \((1,2) \) entries in the representation of \(\widetilde M_n \) given by~\eqref{matrizhat} and~\eqref{matriz1}, we obtain
\begin{align}
\label{eq:representacao}
\Phi_1^n = \overline b - (\overline b + n) |\alpha_{n-1}|^2 && \text{and} &&
\alpha_n = \frac{b+n}{\overline b + n +1} \alpha_{n-1}, && n \in \mathbb N.
\end{align}
The expression~\eqref{eq:representacao} for the \(\alpha_n \) was obtained in~\cite{Ra10}.

Applying Proposition~\ref{propnula}, we get a zero curvature formula.
In fact, multiplying~\eqref{curvatureformula} by \( z ( 1-z ) \), we arrive to
\begin{align*}
z(1-z)T_n^\prime +T_n\widetilde{M}_n-\frac{1-z}{2}T_n-\widetilde{M}_{n+1}T_n=\mathbf{0}_{2}, && n \in \mathbb N 
 ,
\end{align*}
and continue as in Proposition~\ref{propeqfund2} we get
\begin{align*}
\widetilde{M}_{n+1}\bigg\{z(1-z)T_n^\prime -\frac{(1-z)}{2}T_n\bigg\}+\bigg\{z(1-z)T_n^\prime -\frac{(1-z)}{2}T_n\bigg\}\widetilde{M}_{n}=\widetilde{M}_{n+1}^2T_n-T_n\widetilde{M}_n^2.
\end{align*}



We are interested in the differential equations fulfilled by the orthogonal polynomials on the unit circle. Here, we use the Riemann--Hilbert problem approach in order to derive these differential relations.

Multiplying equation~\eqref{Mn} by \( z (1-z) \) we get \( z (1-z) Z_n^\prime (z) = \widetilde M_n Z_n (z) \), and from this, we get
\begin{align*}
 z(1-z)Y_n^{\prime}=\widetilde{M}_nY_n-z(1-z)Y_nC_n^{\prime}C_n^{-1} ,
\end{align*}
which entrywise reads as:
{\small
\begin{align}
 & 
z(1-z)\Phi_n^{\prime}(z)
 =
 \big( -nz+(\overline{b}+n)(1-|\alpha_{n-1}|^2) \big)
 \Phi_n(z)+(\overline{b}+n)(1-|\alpha_{n-1}|^2)\overline{\alpha}_{n-1}\Phi_{n-1}^*(z)
 \label{eqdifjacobi}, 
 \\
\nonumber & z(1-z)G_n^{\prime}(z)
 =
 \big( 
 -bz-(\overline{b}+n)|\alpha_{n-1}|^2
 \big)
 G_n(z)+(\overline{b}+n)(1-|\alpha_{n-1}|^2)\overline{\alpha}_{n-1}G_{n-1}^*(z),
 \\
 \nonumber & z(1-z) \big( \Phi_{n-1}^* \big)^{\prime}(z)
 =
\big(
bz+(\overline{b}+n)|\alpha_{n-1}|^2
\big)
\Phi_{n-1}^*(z)+(\overline{b}+n)\alpha_{n-1}\Phi_{n}(z), 
 \\
\nonumber & z(1-z) \big( G_{n-1}^* \big)^{\prime}(z)
 =
\big(
nz-(\overline{b}+n)(1-|\alpha_{n-1}|^2)
\big)
G_{n-1}^*(z)+(\overline{b}+n)\alpha_{n-1}G_{n}(z).
\end{align}
}


Note that with some manipulations, it can be shown that equation \eqref{eqdifjacobi} is equivalent to the structure relation
\begin{align*}
(z-1)\Phi_{n}^{\prime}(z)=-(\overline{b}+n)(1-|\alpha_{n-1}|^2)\Phi_{n-1}(z)+n\Phi_{n}(z)
\end{align*}
presented in \cite{BRS23}.

Applying Proposition~\ref{propeqdif2}, we can derive second order differential equation for the matrix~\(Y_n \). In fact, we only have to multiply~\eqref{op2orden} by \( z(1-z) \), to obtain
\begin{align*}
z(1-z)Y_n^{\prime\prime}+2Y_n^{\prime} ( z(1-z) C_n^{\prime}C_n^{-1} )+Y_n ( z(1-z) C_n^{\prime\prime} C_n^{-1}) =\Big( z(1-z) {M}_n^{\prime} + \frac{\widetilde{M}_n^2}{z(1-z)} \Big) \, Y_n.
\end{align*}
We can see that
\( \displaystyle z(1-z) {M}_n^{\prime} = \widetilde M_n^\prime - \frac{1}{z} \widetilde M_n + \frac{1}{1-z} \widetilde M_n \)
and
\begin{align*}
 \widetilde M_n = z(1-z) Z_n^\prime Z_n^{-1} =
 z(1-z) \big( Y_n^\prime Y_n^{-1} + Y_n C_n^\prime C_n^{-1} Y_n^{-1} \big).
\end{align*}
Substituting this into the last equation we arrive to
\begin{multline}
\label{opsegundoother}
 z(1-z)Y_n^{\prime\prime}+Y_n^\prime[2z(1-z) \, C_n^\prime C_n^{-1}+(1-2z) \, \mathbf{\operatorname I}_2 ]
 \\
+Y_n \, [z(1-z)C_n^{\prime\prime}C_n^{-1} +(1-2z)C_n^\prime C_n^{-1}] 
 =\Big(
 \widetilde{M}_n^{\prime}+\frac{\widetilde{M}_n^2}{z(1-z)}
 \Big)
 Y_n.
\end{multline}
Now, we will split~\eqref{opsegundoother} in four second order differential equations for \( \Phi_n\), \( \Phi_n^* \), \( G_n\), \( G_n^* \).
From the representation of \(\widetilde M_n \) given in~\eqref{matrizhat}, we get
\begin{align*}
&\widetilde M_n^2 = \big( z | \alpha_{n-1}|^2 (n+b) (\overline b + n )+\frac{1}{4} ( \overline b - z (n+b)+n )^2 \big) \, \mathbf{\operatorname I}_2, \\
&\widetilde M_n^\prime =
\begin{bmatrix}
 -\frac{n+b}{2} & 0 \\
 0 & \frac{n+b}{2} \\
\end{bmatrix},
\end{align*}
and by definition of \( C_n \) in \eqref{Zn}, we have
\begin{align*}
 & 2z(1-z)C_n^{\prime}C_n^{-1}+(1-2z)\mathbf{\operatorname I}_2=
\begin{bmatrix}
1-n(1-z)-(2+b)z-\overline{b} & 0\\
0 & 1+n+(b-2-n)z+\overline{b}
\end{bmatrix} ,
 \\
 & \begin{multlined}
z(1-z)C_n^{\prime\prime}C_n^{-1}+(1-2z)C_n^{\prime}C_n^{-1}
 = \begin{bmatrix}
 -\frac{1}{4} (n-b-2) (n-b) & 0 \\
 0 & -\frac{1}{4} (n-b) (n-b+2) \\
\end{bmatrix} \\
+ \frac{(\overline{b}+n)^2}{4z} \mathbf{\operatorname I}_2
+ \frac{\operatorname{Re} (b)^2}{1-z} \mathbf{\operatorname I}_2 .
\end{multlined}
\end{align*}
Moreover,
it can be seen
that 
\begin{align*}
\frac{\widetilde M_n^2}{z(1-z)} 
 =
\Big( -\frac{1}{4} (n+b)^2 +\frac{(\overline{b}+n)^2}{4z} +\frac{ - 
\operatorname{Im}(b)^2 + (b+n)(\overline{b}+n) | \alpha_{n-1}| ^2}{1-z} \Big) \, \mathbf{\operatorname I}_2
\end{align*} 
Now, from the second identity in~\eqref{eq:representacao} it can be proven that
\begin{align*}
 (n+ b ) (\overline b +n ) |\alpha_{n-1}|^2 = | b| ^2 ,
\end{align*}
and so, substituting all these matrices into \eqref{opsegundoother}, we obtain
\begin{align*}
&
z(1-z)\Phi_n^{\prime\prime}(z)+[(n-b-2)z+(1-n-\overline{b})]\Phi_n^{\prime}(z)+n(1+b)\Phi_n(z)=0,
 \\ 
&
z(1-z)G_n^{\prime\prime}(z)+[(b-n-2)z+(1+n+\overline{b})]G_n^{\prime}(z)+b(1+n)G_n(z)=0,\\
& 
 z(1-z) \big( \Phi_{n-1}^* \big)^{\prime\prime}(z)+\big((n-b-2)z+(1-n-\overline{b})\big) \big( \Phi_{n-1}^* \big)^{\prime}(z)+b(n-1)\Phi_{n-1}^*(z)=0,
 \\ 
&
 z(1-z) \big(G_{n-1}^*\big)^{\prime\prime}(z)+\big((b-n-2)z+(1+n+\overline{b})\big) \big( G_{n-1}^* \big)^{\prime}(z)+n(b-1)G_{n-1}^*(z)=0,
\end{align*}
which are hypergeometric differential equations for \( \Phi_n \), \( \Phi_n^* \), \( G_n \), and \( G_n^* \).

\section*{Acknowledgments}

\begin{enumerate}[\rm\( \bullet \)]
 \item
 AB
acknowledges Centre for Mathematics of the University of Coimbra 
(funded by the Portuguese Government through FCT/MCTES, 
UID/00324/2025.
\item
 AF
acknowledges the CIDMA Center for Research and Development in Mathematics and Applications
(University of Aveiro) and the Portuguese Foundation 
for Science and Technology (FCT) for their support within
project, 
UID/04106/2025.

\item
KR acknowledges the CAPES and PROPG/UNESP for support by doctoral sandwich students grants. The work was done while Karina Rampazzi was visiting Ana Foulqui\'e at the Department of Mathematics of University of Aveiro (UA). Their stay at UA for a period of six months during 2023. This author is extremely grateful to UA for receiving all the necessary support to undertake this research.

\end{enumerate}

\section*{Declarations}

\begin{enumerate}[\rm 1)]
 \item \textbf{Conflict of interest:} 
 The authors declare no conflict of interest.
 \item \textbf{Ethical approval:} 
 Not applicable.
 \item \textbf{Contributions:} 
 All the authors have contribute equally.
 \item \textbf{Data availability:} 
 This paper has no associated data.
\end{enumerate}

\end{document}